\documentclass[11pt,reqno]{amsart}
\usepackage[usenames]{color}
\usepackage{amsmath, amssymb, latexsym, multicol, amsthm, color, enumitem,comment}
\usepackage[all]{xy}
\usepackage{cancel}
\newtheorem{theorem}{Theorem}
\newtheorem{lemma}[theorem]{Lemma}

\newtheorem{proposition}[theorem]{Proposition}

\theoremstyle{remark}
\newtheorem*{remark}{Remark}

\newtheorem{example}[theorem]{Example}
\newcommand{\Z}{\mathbb{Z}}
\newcommand{\Q}{\mathbb{Q}}
\newcommand{\N}{\mathbb{N}}

\newcommand{\R}{\mathbb{R}}
\newcommand{\C}{\mathbb{C}}

\newcommand{\im}{\operatorname{Im}}
\renewcommand{\H}{\mathbb{H}}

\newcommand{\SL}{\operatorname{SL}}
\newcommand{\sgn}{\operatorname{sgn}}

\numberwithin{equation}{section}
\numberwithin{theorem}{section}

\title[Bruinier--Funke pairing and unary theta functions]{The Bruinier--Funke pairing and the orthogonal complement of unary theta functions}

\author{Ben Kane}
\address{Ben Kane, Department of Mathematics, University of Hong Kong, Pokfulam, Hong Kong.}
\email{bkane@hku.hk}

\author{Siu Hang Man}
\address{Siu Hang Man, Department of Mathematics, University of Hong Kong, Pokfulam, Hong Kong.}
\email{der.gordox@gmail.com}

\thanks{The research of the first author was supported by grant project numbers 27300314, 17302515, and 17316416 of the Research Grants Council. }

\begin{document}

\begin{abstract}
We describe an algorithm for computing the inner product between a holomorphic modular form and a unary theta function, in order to determine whether the form is orthogonal to unary theta functions without needing a basis of the entire space of modular forms and without needing to use linear algebra to decompose this space completely.
\end{abstract}
\date{\today}
\maketitle

\section{Introduction}

In this paper, we are interested in the decomposition of holomorphic modular forms.  Suppose that $f$ is a weight $3/2$ holomorphic modular form on some congruence subgroup $\Gamma$.  One can decompose $f$ into an Eisenstein series component $E$, a sum $\Psi$ of (cuspidal) unary theta functions (see \eqref{eqn:unarydef} for the definition), and a cusp form $g$ in the orthogonal complement of unary theta functions.  This is an orthogonal splitting with respect to the usual Petersson inner product, since the Eisenstein series is orthogonal to cusp forms.  It is thus natural to try to compute the individual pieces.  The Eisenstein series component may be computed by determining the growth of $f$ towards the cusps.  Furthermore, its Fourier coefficients may be explicitly computed, and these generally constitute the main asymptotic term of the Fourier coefficients of $f$.  In a number of combinatorial applications, this is quite useful in determining the overall growth of the coefficients of $f$.  For example, if $f$ is the generating function for the number of representations by a ternary quadratic form $Q$, then the coefficients of the Eisenstein series count the number of local representations, and the fact that this is (usually) the main asymptotic term implies an equidistribution result about the representations of integers in the \begin{it}genus\end{it} of $Q$ (i.e., those quadratic forms which are locally equivalent to $Q$).  This equidistribution result does not always hold, however; the coefficients of $\Psi$ grow as fast as the coefficients of $E$ within their support, although they are only supported in finitely many square classes (known on the algebraic side of the theory of quadratic forms as spinor exceptional square classes).  Using an upper bound of Duke \cite{Duke} for the coeffients of $g$, Duke and Schulze-Pillot \cite{DS-P90} combined these ideas to conclude an equidistribution result for the primitive representations by every element of the genus away from these spinor exceptional square classes.

It is natural to ask whether similar results hold true when the quadratic form is replaced with a \begin{it}totally positive quadratic polynomial\end{it} (i.e., a form constructed as a linear combination of a positive-definite integral quadratic form, linear terms, and the unique constant such that the quadratic polynomial only represents non-negative integers and represents zero).  One such example is sums of polygonal numbers.  For $n\in\Z$, the \begin{it}$n$th generalized $m$-gonal number\end{it} is 
\[
p_m(n):=\frac{(m-2)n^2-(m-4)n}{2},
\]
and for $a,b,c\in\N$ we investigate sums of the type 
\[
P(x,y,z)=P_{a,b,c}(x,y,z):=a p_{m}(x)+ bp_{m}(y)+cp_m(z),
\]
where $x,y,z\in\Z$.  We consider $a$, $b$, and $c$ to be fixed and vary $x$, $y$, and $z$.  We package $P$ into a generating function
\[
\sum_{x,y,z\in\Z} e^{2\pi i P(x,y,z) \tau}
\]
with $\tau\in\H:=\{ \alpha\in\C: \im(\alpha)>0\}$; this is known as the \begin{it}theta function\end{it} for $P$.  We may then investigate the Fourier coefficients of this theta function in order to attempt to understand which integers are represented by $P$.  It is actually more natural to complete the square to rewrite 
\[
p_m(x)=\frac{\left(2(m-2)x-(m-4)\right)^2}{8(m-2)} -\frac{(m-4)^2}{8(m-2)}
\]
Adding an appropriate constant, we obtain a theta function for a shifted lattice $L+\nu$, where $\nu\in \Q L$ inside a quadratic space with associated quadratic norm $Q$; quadratic forms are simply the case when $\nu=0$ (or equivalently, $\nu\in L$).  These theta functions are again modular forms and the unary theta functions govern whether the local-to-global principle fails finitely or infinitely often.
\begin{theorem}\label{thm:shiftedlattice}
 Suppose that $L$ is a ternary positive-definite lattice and $\nu$ is a vector in the associated quadratic space over $\Q$.  Suppose further that the congruence class $(M\Z+r)\cap\N_0$ is primitively represented locally by the associated quadratic form $Q$ on $L+\nu$ and denote by $a_{L+\nu}(Mn+r)$ the number of vectors of length $Mn+r$ in $L+\nu$ (i.e., the number of $\mu\in L+\nu$ for which $Q(\mu)=Mn+r$).  If 
\[
\Theta_{L+\nu}(\tau):=\sum_{\mu \in L+\nu} e^{2\pi i Q(\mu) \tau}
\]
 is orthogonal to unary theta functions, then 
\[
\left\{n\in\Z: \not\exists \mu\in L+\nu,\ Q(\mu)=Mn+r \right\}
\]
is finite.
\end{theorem}
\begin{remark}
If $\Theta_{L+\nu}$ is orthogonal to unary theta functions for every $L+\nu$ in a given genus, then one obtains an equidistribution result for representations of $Mn+r$ (for $n$ sufficiently large, but with an ineffective bound) across the entire genus in the same manner as for the case of quadratic forms. 
\end{remark}
There are a number of cases where Theorem \ref{thm:shiftedlattice} has been employed to show that certain quadratic polynomials $P$ are \begin{it}almost universal\end{it} (i.e., they represent all but finitely many integers).  In the case of triangular numbers (that is to say, $m=3$), the first author and Sun \cite{KS08} obtained a near-classification which was later fully resolved by Chan--Oh \cite{CO09}; further classification results about sums of triangular numbers and squares were completed by Chan--Haensch \cite{CH11}.  More recently, the case $a=b=c=1$ with arbitrary $m$ was considered by Haensch and the first author \cite{HaenschKane}.  In \cite{HaenschKane}, a number of almost universality results are obtained by taking advantage of the fact that the structure of modular forms may be used to determine that certain congruence classes are not in the support of the coefficients of all of the unary theta functions in the same space, and hence directly obtaining the orthogonality needed for Theorem \ref{thm:shiftedlattice}.  This was generalized by the second author and Mehta \cite{ManMehta} to include many more cases of $a,b,c$ where the same phenomenon implies orthogonality.  We next consider a case which does not immediately follow from this approach.
\begin{proposition}\label{prop:octagonal}
Every sufficiently large positive integer may be written in the form $p_{8}(x)+3p_8(y)+3p_8(z)$ with $x,y,z\in\Z$.  In other words, $p_{8}(x)+3p_8(y)+3p_8(z)$ is almost universal.  
\end{proposition}
In order to show Proposition \ref{prop:octagonal}, we use Theorem \ref{thm:shiftedlattice} and show that the theta function $\Theta_{L+\nu}$ associated to $p_8(x)+3p_8(y)+3p_8(z)$ is orthogonal to all unary theta functions.  One can numerically compute the inner product with unary theta functions directly from the definition as an integral over a fundamental domain of $\SL_2(\Z)\backslash\H$ or use a method called unfolding to write it as infinite sums involving products of the Fourier coefficients of $\Theta_{L+\nu}$ and those of the unary theta functions.  However, this is not sufficient for our purposes, since we need to algebraically verify that the inner product is indeed zero and the first method is only a numerical approximation while the second method yields an infinite sum.  Since the associated space of modular forms is finite-dimensional and there is a natural orthogonal basis of Hecke eigenforms, one can decompose the space explicitly to determine whether this orthogonality holds, but the linear algebra involved is usually computationally expensive and is not feasible in many cases.  We hence use a pairing of Bruinier and Funke \cite{BruinierFunke} to rewrite the inner product as a finite sum.  The basic idea is to use Stokes' Theorem to rewrite the inner product as a (finite) linear combination of products of the Fourier coefficients of $\Theta_{L+\nu}$ and coefficients of certain ``pre-images'' of the unary theta functions under a natural differential operator.  In order to find these pre-images, we employ work of Zwegers \cite{ZwegersThesis}, who showed that these pre-images are related to the mock theta functions of Ramanujan.  

The paper is organized as follows.  In Section \ref{sec:prelim}, we give some preliminary information about modular forms and harmonic Maass forms. In Section \ref{sec:inner}, we describe how to compute the inner product using the Bruinier--Funke pairing and construct explicit pre-images of unary theta functions using \cite{ZwegersThesis} (see Theorem \ref{thm:Zwegersfull}).  
Finally, in Section \ref{sec:applications}, we prove Theorem \ref{thm:shiftedlattice} and Proposition \ref{prop:octagonal}.

\section{Preliminaries}\label{sec:prelim}
We recall some results about modular forms and harmonic Maass forms.  
\subsection{Basic definitions}
Let $\H$ denote the \begin{it}upper half-plane\end{it}, i.e., those $\tau=u+iv\in \C$ with $u\in\R$ and $v>0$.  The matrices $\gamma=\left(\begin{smallmatrix} a&b\\ c&d\end{smallmatrix}\right)\in\SL_2(\Z)$ (the space of two-by-two integral matrices with integer coefficients and determinant $1$) act on $\H$ via \begin{it}fractional linear transformations\end{it} $\gamma \tau:=\frac{a\tau+b}{c\tau+d}$.  For 
\[
j(\gamma,\tau):=c\tau+d,
\]
a \begin{it}multiplier system\end{it} for a subgroup $\Gamma\subseteq \SL_2(\Z)$ and \begin{it}weight\end{it} $r\in \R$ is a function $\nu:\Gamma\mapsto \C$ such that for all $\gamma,M\in\Gamma$ (cf. \cite[(2a.4)]{Pe1})
\[
\nu(M \gamma) j(M\gamma,\tau)^r = \nu(M)j(M,\gamma \tau)^r \nu(\gamma)j(\gamma,\tau)^r.
\]
The \begin{it}slash operator\end{it} $|_{r,\nu}$ of weight $r$ and multiplier system $\nu$ is then 
\[
f|_{r,\nu}\gamma (\tau):=\nu(\gamma)^{-1} j(\gamma,\tau)^{-r} f(\gamma \tau).
\]
A \begin{it}harmonic Maass form\end{it} of weight $r\in\R$ and multiplier system $\nu$ for $\Gamma$ is a function $f:\H\to\C$ satisfying the following criteria:
\noindent

\noindent
\begin{enumerate}[leftmargin=*]
\item
The function $f$ is annihilated by the weight $r$ hyperbolic Laplacian 
\[
\Delta_r:=-\xi_{2-r}\circ \xi_{r},
\]
where 
\begin{equation}\label{eqn:xidef}
\xi_{r}:=2iv^r\overline{\frac{\partial}{\partial \overline{\tau}}}.
\end{equation}

\item
For every $\gamma\in\Gamma$, we have 
\begin{equation}\label{eqn:modularity}
f|_{r,\nu}\gamma= f.
\end{equation}
\item
The function $f$ exhibits at most linear exponential growth towards every \begin{it}cusp\end{it} (i.e., those elements of $\Gamma\backslash(\Q\cup\{i\infty\})$).  This means that at each cusp $\varrho$ of $\Gamma\backslash \H$, the Fourier expansion of the function $f_{\varrho}(\tau):=f|_{r,\nu}\gamma_{\varrho}(\tau)$ has at most finitely many terms which grow, where $\gamma_{\varrho}\in \SL_2(\Z)$ sends $i\infty$ to $\varrho$.  
\end{enumerate}
If $f$ is holomorphic and the Fourier expansion at each cusp is bounded, then we call $f$ a holomorphic modular form.  Furthermore, if $f$ is a holomorphic modular form and vanishes at every cusp (i.e., the limit $\lim_{\tau\to i\infty} f_{\varrho}(\tau)=0$), then we call $f$ a \begin{it}cusp form\end{it}.  

\subsection{Half-integral weight forms}
We are particularly interested in the case where $r=k+1/2$ with $k\in\N_0$ and, in the example given in Theorem \ref{thm:shiftedlattice} that motivates this study we may choose $\Gamma$ to be an intersection between the groups
\begin{align*}
\Gamma_0(M)&:=\left\{ \left(\begin{matrix}a&b\\ c&d\end{matrix}\right)\in\SL_2(\Z): M\mid c\right\},\\
\Gamma_1(M)&:=\left\{ \left(\begin{matrix}a&b\\ c&d\end{matrix}\right)\in\SL_2(\Z): M\mid c, a\equiv d\equiv 1\pmod{M}\right\}
\end{align*}
for some $M\in\N$ divisible by $4$.  The multiplier system we are particularly interested in is given in \cite[Proposition 2.1]{Shimura}, although we do not need the explicit form of the multiplier for this paper.

If $T^N\in \Gamma$ with $T:=\left(\begin{smallmatrix} 1&1\\ 0 &1\end{smallmatrix}\right)$, then by \eqref{eqn:modularity} we have $f(\tau+N)=f(\tau)$, and hence $f$ has a Fourier expansion ($c_{f}(v;n)\in\C$)
\begin{equation}\label{eqn:fexp}
f(\tau)=\sum_{n\gg -\infty} c_{f}(v;n) e^{\frac{2\pi i n \tau}{N}}.
\end{equation}
Moreover, $f$ is meromorphic if and only if $c_f(v;n)=c_f(n)$ is independent of $v$.  For holomorphic modular forms, an additional restriction $n\geq 0$ follows from the fact that $f$ is bounded as $\tau\to i\infty$.  There are similar expansions at the other cusps.  One commonly sets $q:=e^{2\pi i \tau}$ and associates the above expansion with the corresponding formal power series, using them interchangeably unless explicit analytic properties of the function $f$ are required.

\subsection{Theta functions for quadratic polynomials}
In \cite[(2.0)]{Shimura}, Shimura defined theta functions associated to lattice cosets $L+\nu$ (for a lattice $L$ of rank $n$) and polynomials $P$ on lattice points.  Namely, he defined
\[
\Theta_{L+\nu,P}(\tau):=\sum_{\boldsymbol{x}\in L+\nu} P(\boldsymbol{x}) q^{Q(\boldsymbol{x})},
\]
where $Q$ is the quadratic map in the associated quadratic space.  We omit $P$ when it is trivial.  In this case, we may write $r_{L+\nu}(\ell)$ for the number of elements in $L+\nu$ of norm $\ell$ and we get
\begin{equation}\label{eqn:thetadef}
\Theta_{L+\nu}(\tau)=\sum_{\ell\geq 0} r_{L+\nu}(\ell) q^{\ell}.
\end{equation}
Shimura then showed (see \cite[Proposition 2.1]{Shimura}) that $\Theta_{L+\nu}$ is a modular form of weight $n/2$ for $\Gamma=\Gamma_0(4N^2)\cap \Gamma_1(2N)$ (for some $N$ which depends on $L$ and $\nu$) and a particular multiplier. Note that we have taken $\tau\mapsto 2N\tau$ in Shimura's definition.  To show the modularity properties, for $\gamma=\left(\begin{smallmatrix}a&b\\ c&d\end{smallmatrix}\right)\in \Gamma$, we compute
\begin{equation}\label{eqn:flipNgamma}
2N\gamma(\tau)=2N\frac{a\tau+b}{c\tau+d} = \frac{a(2N\tau) + 2Nb}{\frac{c}{2N} (2N\tau)+d} = \left(\begin{matrix} a & 2Nb\\ \frac{c}{2N} & d\end{matrix}\right) (2N\tau).
\end{equation}
Since $\gamma\in \Gamma$, we have 
\[
\left(\begin{matrix} a & 2Nb\\ \frac{c}{2N} & d\end{matrix}\right)\in \Gamma(2N):=\left\{ \gamma= \left(\begin{matrix}a&b\\ c&d\end{matrix}\right)\in\SL_2(\Z): \gamma\equiv I_{2}\pmod {N}\right\}\subset \Gamma_{1}(2N),
\]
so we may then use \cite[Proposition 2.1]{Shimura}.  Specifically, the multiplier is the same multiplier as $\Theta^3$, where $\Theta(\tau):=\sum_{n\in\Z} q^{n^2}$ is the classical Jacobi theta function.  

We only require the associated polynomial in one case.  Namely, for $n=1$ and $P(x)=x$, we require the \begin{it}unary theta functions\end{it}  (see \cite[(2.0)]{Shimura} with $N\mapsto N/t$, $P(m)=m$, $A=(N/t)$, 
and $\tau\mapsto 2N\tau$)
\begin{equation}\label{eqn:unarydef}
\vartheta_{h,t}(\tau)=\vartheta_{h,t,N}(\tau):=\sum_{\substack{r\in\Z\\ r\equiv h\pmod{\frac{2N}{t}}}} r q^{t r^2},
\end{equation}
where $h$ may be chosen modulo $2N/t$ and $t$ is a squarefree divisor of $2N$.  These are weight $3/2$ modular forms on $\Gamma_0(4N^2)\cap \Gamma_1(2N)$ with the same multliplier system as $\Theta_{L+\nu}$.

\section{The Bruinier--Funke pairing}\label{sec:inner}
In this section, we describe how to compute the inner product with unary theta functions.  We again begin by noting the decomposition of a weight $3/2$ modular form $f$ as
\[
f=E+\Psi+g,
\]
where $E$ is an Eisenstein series, $\Psi$ is a linear combination of unary theta functions, and $g$ is a cusp form in the orthogonal complement of unary theta functions.  Since the decomposition above is an orthogonal splitting with respect to the Petersson inner product, one may instead compute the inner product 
\[
\left<f,\Theta_j\right>
\]
for each unary theta function $\Theta_j$.  Recall that Petersson's classical definition of the inner product between two holomorphic modular forms $f$ and $h$ (for which $fh$ is cuspidal) is (here and throughout $\tau=u+iv$) 
\[
\left<f,h\right>:=\frac{1}{\left[\SL_2(\Z):\Gamma\right]} \int_{\Gamma\backslash\H} f(\tau) \overline{h(\tau)} v^{\frac{3}{2}} \frac{du dv}{v^2},
\]
where $\left[\SL_2(\Z):\Gamma\right]$ denotes the index of $\Gamma$ in $\SL_2(\Z)$.  While one may be able to approximate the integral well numerically, we are interested in obtaining a precise (algebraic) formula for the inner product (and hence an explicit formula for $\Psi$).  In order to do so, we rely on a formula of Bruinier and Funke (see \cite[Theorem 1.1 and Proposition 3.5]{BruinierFunke}) known as the \begin{it}Bruinier--Funke pairing\end{it}.  The basic premise is to use Stokes' Theorem in order to compute the inner product in a different way.  Suppose that we have a preimage $\mathcal{H}$ under the operator $\xi_{1/2}$, where
\[
\xi_{\kappa}:=2iv^{\kappa} \overline{\frac{\partial}{\partial \overline{\tau}}}
\]
is a differential operator which sends functions satisfying weight $\kappa$ modularity to functions satisfying weight $2-\kappa$ modularity.  Note that since $h$ is holomorphic and $\xi_{1/2}(\mathcal{H})=h$, the fact that the kernel of $\xi_{2-\kappa}$ is holomorphic functions implies that the function $\mathcal{H}$ is necessarily annihilated by the weight $\kappa$ hyperbolic Laplacian (for $\kappa=1/2$) 
\[
\Delta_{\kappa}=-\xi_{2-\kappa}\circ\xi_{\kappa}.
\]
If we further impose that $\mathcal{H}$ is modular of weight $\kappa$ on $\Gamma$ and has certain restrictions on its singularities in $\Gamma\backslash(\H\cup \Q\cup\{i\infty\})$ (see Section \ref{sec:prelim} for further details), then we obtain a harmonic Maass form.  Due to the fact that $\Gamma$ is a congruence subgroup, it contains $T^{N}$ for some $N$, where $T:=\left(\begin{smallmatrix}1&1\\ 0 &1\end{smallmatrix}\right)$.  Similarly, if $\gamma_{\varrho}\in \SL_2(\Z)$ sends $i\infty$ to a cusp $\varrho$, then $T^{N_{\varrho}}$ is contained in $\gamma_{\varrho}^{-1}\Gamma \gamma_{\varrho}$ for some $N_{\varrho}\in\N$;  here $N_{\varrho}$ is known as the cusp width at $\varrho$.  Using this, one can show that it has a Fourier expansion around each cusp $\varrho$ of $\Gamma$ of the shape
\[
\mathcal{H}_{\varrho}(\tau)=\sum_{n\in\Z} c_{\mathcal{H},\varrho}(v;n) e^{\frac{2\pi i n\tau}{N_{\varrho}}},
\]
for some $c_{\mathcal{H},\varrho}(y;n)\in\C$, and where $\mathcal{H}_{\varrho}:=\mathcal{H}|_{\kappa} \gamma_{\varrho}$ is the expansion around $\varrho$.  Note however, that since $\mathcal{H}$ is not holomorphic, the Fourier coefficients may depend on $v$.  Solving the differential equation $\Delta_{\kappa}(\mathcal{H})=0$ termwise yields a natural splitting of the Fourier expansion into holomorphic and non-holomorphic parts, namely 
\[
\mathcal{H}_{\varrho}(\tau)=\mathcal{H}_{\varrho}^+(\tau) + \mathcal{H}_{\varrho}^-(\tau)
\]
with 
\begin{align*}
\mathcal{H}_{\varrho}^+(\tau)&=\sum_{n\gg -\infty} c_{\mathcal{H},\varrho}^+(n) e^{\frac{2\pi i n\tau}{N_{\varrho}}}\\
\mathcal{H}_{\varrho}^-(\tau)&=c_{\mathcal{H},\varrho}^-(0)v^{2-\kappa} + \sum_{\substack{n\ll \infty\\ n\neq 0}} c_{\mathcal{H},\varrho}^-(n) \Gamma\left(2-\kappa, -\frac{4\pi n v}{N_{\varrho}}\right) e^{\frac{2\pi i n\tau}{N_{\varrho}}},
\end{align*}
where now the coefficients are independent of $v$.  It is these Fourier coefficients which are used by Bruinier and Funke to compute the inner product explicitly in \cite[Proposition 3.5]{BruinierFunke}.  To state their formula, let $\mathcal{S}_{\Gamma}$ denote the set of cusps and write 
\[
f_{\varrho}(\tau)=\sum_{n\geq 0} c_{f,\varrho}(n)e^{\frac{2\pi i n\tau}{N_{\varrho}}}.
\]
\begin{theorem}[Bruinier--Funke]\label{thm:BF}
We have 
\[
\left< f,h\right> = \frac{1}{\left[\SL_2(\Z): \Gamma\right]} \sum_{\varrho\in \mathcal{S}_{\Gamma}}\sum_{n\geq 0} c_{f,\varrho}(n) c_{\mathcal{H},\varrho}^+(-n).
\]
\end{theorem}
Theorem \ref{thm:BF} is algebraic, precise, and is actually a finite sum since there are only finitely many $n$ for which $c_{\mathcal{H},\varrho}^+(-n)\neq 0$, allowing one to explicitly compute the inner product.  We will assume that sufficiently many Fourier coefficients of $f$ are known, or in other words the input to our algorithm will be the Fourier coefficients $c_{f,\varrho}(n)$ and the function $h$, which in our case will be a unary theta function.  The assumption that the expansions are known at every cusp may at first seem to be a somewhat strong assumption, since in combinatorial applications we often only know the expansion at one cusp.  However, when $f=\Theta_{L+\nu}$ is the theta function for a shifted lattice, Shimura \cite{Shimura} has computed the modularity properties for all of $\SL_2(\Z)$ and one obtains modularity for $\SL_2(\Z)$ in a vector-valued sense, where the components of the vector are the functions $f_{\varrho}$.  In other words, given just the theta function $f$, one can determine the functions $f_{\varrho}$ as long as one can write $\gamma_{\varrho}$ explicitly in terms of the generators $S:=\left(\begin{smallmatrix} 0 &-1\\ 1&0\end{smallmatrix}\right)$ and $T$ of $\SL_2(\Z)$.  Although this rewriting is well-known, we provide the details for the convenience of the reader.
\begin{lemma}\label{lem:matrixrewrite}
Given $\varrho=a/c$, there is an algorithm to determine $\gamma_{\varrho}\in\SL_2(\Z)$ explicitly in terms of $S$ and $T$.
\end{lemma}  
\begin{proof}
First, we need to construct $\gamma_{\varrho}$ for which $\gamma_{\varrho}(i\infty)=a/c$.  In other words, we want a matrix $\left(\begin{smallmatrix} a &b\\ c&d\end{smallmatrix}\right)\in \SL_2(\Z)$.  Since $ad-bc=1$ and $a$ and $c$ are necessarily prime, we see that $b$ and $-c$ are precisely the coefficients from Bezout's theorem.  We next construct the sequence of $S$ and $T$ recursively as follows.  

Let $\gamma_0:=\gamma_{\varrho}$.  At step $j+1$ (with $j\in\N_0$) we will construct $\gamma_{j+1}$ inductively/recursively from $\gamma_{j}$ by multiplying either by $S$ or by $ST^m$ for some $m\in\Z$, and eventually obtain $\gamma_{\ell}=\pm T^m$ for some step $\ell$ and $m\in\Z$.  Suppose that 
\[
\gamma_{j}=\left(\begin{matrix} a_j & b_j\\ c_j & d_j\end{matrix}\right).
\]
If $c_j=0$, then $a_j=d_j=\pm 1$ and $\ell=j$ with $\gamma_j=\pm T^{\pm b_j}$, and reversing back through the recursion gives the expansion of $\gamma_0$ in terms of $S$ and $T$, so we are done.  

If $c_j\neq 0$, then we choose $r\in \Z$ such that $|a_j+r c_j|$ is minimal (if there are two choices, i.e, if $a_j+rc_j=c_j/2$ for some $r$, then we take this choice of $r$).  We then set 
\[
\gamma_{j+1}:=ST^r \gamma_{j-1} = S\left(\begin{matrix} a_j+r c_j & b_j+rd_j\\ c_j & d_j\end{matrix}\right)=\left(\begin{matrix}-c_j  &-d_j \\ a_j+r c_j & b_j+rd_j\end{matrix}\right). 
\]
Note that $|a_j+rc_j|\leq |c_j|/2$ by construction, so the entry in the lower-left corner is necessarily smaller at step $j+1$ than it was at step $j$.  Therefore the algorithm will halt after a finite number of steps.
\end{proof}

In order to determine the inner product $\left<f,h\right>$, it remains to compute the preimage $\mathcal{H}$ and compute its Fourier expansion.  Luckily, motivated by Ramanujan's mock theta functions, Zwegers \cite{ZwegersThesis} constructed pre-images of the unary theta functions using a holomorphic function $\mu$ which he ``completed'' to obtain a harmonic Maass form (actually, he is even able to view his completed object as a non-holomorphic Jacobi form, and one obtains the pre-images of unary theta functions by plugging in elements of $\Q+\Q\tau$ for the elliptic variable $z$).  Choosing $z$ to be an appropriate element of $\Q+\Q\tau$, one may compute the expansions at all cusps by viewing Zwegers's function as a component of a vector-valued modular form.  As a first example, Zwegers himself computed the corresponding vector when the unary theta function is given by 
\[
\Theta_0(\tau):=\sum_{n\in\Z} \left(n+\frac{1}{6}\right) e^{3\pi i \left(n+\frac{1}{6}\right)^2\tau}.
\]
This is related to the third order mock theta function $f(q)$, and played an important role in Bringmann and Ono's \cite{BringmannOnoInvent} proof of the Andrews--Dragonette conjecture.  One may find the full transformation properties listed in \cite[Theorem 2.1]{BringmannOnoInvent}.  Specifically, let 
\[
f(q):=1+\sum_{n=1}^{\infty} \frac{q^{n^2}}{(1+q)^2\!\left(1+q^2\right)^2\cdots\!\left(1+q^n\right)^2}
\]
and 
\[
\omega(q):=\sum_{n=0}^{\infty} \frac{q^{n^2+2n}}{(1-q)^2\!\left(1-q^3\right)^2\cdots\!\left(1-q^{2n+1}\right)^2}.
\]
Setting ($q:=e^{2\pi i \tau}$)
\[
F(\tau)=\left(F_0(\tau),F_1(\tau),F_2(\tau)\right)^T:=\!\left( q^{-\frac{1}{24}}f(q), 2q^{\frac{1}{3}}\omega\!\left(q^{\frac{1}{2}}\right), 2q^{\frac{1}{3}}\omega\!\left(-q^{\frac{1}{2}}\right)\right)^T, 
\]
we have the following.
\begin{theorem}[Zwegers \cite{ZwegersThesis}]\label{thm:Zwegers}
There is a vector-valued harmonic Maass form $\mathcal{H}=\!\left(\mathcal{H}_0,\mathcal{H}_1,\mathcal{H}_2\right)^T$ whose meromorphic part is $F$ (component-wise).  The harmonic Maass form satisfies 
\[
\xi_{\frac{1}{2}}\left(\mathcal{H}_0\right)= \Theta_0
\]
and the modularity properties for $\SL_2(\Z)$ given by 
\begin{align*}
\mathcal{H}(\tau+1) &= \left(\begin{matrix}\zeta_{24}^{-1} & 0 & 0\\ 0& 0&\zeta_3\\ 0&\zeta_3&0\end{matrix}\right)\mathcal{H}(\tau),\\
\mathcal{H}\left(-\frac{1}{\tau}\right) &= \sqrt{-i\tau}\left(\begin{matrix}0 & 1 & 0\\ 1& 0&0\\ 0&0&-1\end{matrix}\right)\mathcal{H}(\tau),
\end{align*}
where $\zeta_n:=e^{2\pi i/n}$.  
\end{theorem}
 Pre-images of a more general family of unary theta functions were investigated by Bringmann and Ono in \cite{BringmannOnoAnnals}; these are connected to Dyson's rank for the partition function, and the modularity of the relevant functions is given in \cite[Theorem 1.2]{BringmannOnoAnnals}, with the full vector-valued transformation properties given in \cite[Theorem 2.3]{BringmannOnoAnnals}.  

Theorem \ref{thm:Zwegers} is the first case of a much more general theorem which follows by combining the results in Zwegers's thesis \cite{ZwegersThesis}.  To describe this result, for $a,b\in\C$ and $\tau\in\H$, define the holomorphic function
\[
\mu(a,b;\tau):=\frac{e^{\pi i a}}{\theta(b;\tau)}\sum_{n\in\Z} \frac{(-1)^ne^{\pi i\left(n^2+n\right)\tau + 2\pi i nb}}{1-e^{2\pi i n\tau+2\pi i a}},
\]
and also define the real-analytic function
\[
R(a;\tau):=\sum_{\nu\in\frac{1}{2}+\Z}\left(\sgn(\nu)-E\left(\left(\nu+\frac{\im(a)}{v}\right)\sqrt{2v}\right)\right)(-1)^{\nu-\frac{1}{2}}e^{-\pi i \nu^2\tau -2\pi i a\nu}, 
\]
where $\sgn(x)$ is the usual sign function,
\[
\theta(z;\tau):=\sum_{\nu\in \frac{1}{2}+\Z} e^{\pi i \nu^2\tau+2\pi i \nu\left(z+\frac{1}{2}\right)},
\]
and
\[
E(z):=\sgn(z)\left(1-\beta\left(z^2\right)\right)
\]
with (for $x\in \R_{\geq 0}$)
\[
\beta(x):=\int_{x}^{\infty} t^{-\frac{1}{2}} e^{-\pi t} dt.
\]
One then defines 
\begin{equation}\label{eqn:tildemudef}
\widetilde{\mu}(a,b;\tau):=\mu(a,b;\tau)+\frac{i}{2} R(a-b;\tau).
\end{equation}
The function $\widetilde{\mu}$ is essentially a weight $1/2$ harmonic Maass form.  
\rm
\begin{theorem}\label{thm:Zwegersfull}
For $h,t,N\in\N$ with $t\mid 2N$, the function 
\[
\mathcal{F}_{h,t,N}(\tau):=-e^{-2\pi i  \left(h-\frac{N}{t}\right)^2\tau }\widetilde{\mu}\left( \frac{ht-N}{2N}\frac{8N^2\tau}{t^2},-\frac{1}{2}; \frac{8N^2\tau}{t^2}\right)
\]
is a weight $1/2$ harmonic Maass form on $\Gamma:=\Gamma_1(4N/t)\cap \Gamma_0(16N^2/t^2)$ with some multiplier system.  Furthermore, it satisfies 
\[
\xi_{\frac{1}{2}}\left(\mathcal{F}_{h,t,N}\right) = \vartheta_{h,t,N}(\tau).  
\]
\end{theorem}
\begin{proof}
The modularity properties of $\mathcal{F}_{h,t,N}$ follow by \cite[Theorem 1.11]{ZwegersThesis}.  In particular, for $\gamma=\left(\begin{smallmatrix}a&b\\ c &d\end{smallmatrix}\right)\in \Gamma$ and $\gamma'=\left(\begin{smallmatrix} a & 8N^2 b/t^2\\ ct^2/(8N^2) & d\end{smallmatrix}\right)$, a change of variables in \cite[Theorem 1.11 (2)]{ZwegersThesis} together with \eqref{eqn:flipNgamma} 
implies that (with $v(\gamma'):=\eta(\gamma'\tau)/(j(\gamma',\tau)\eta(\tau))$ denoting the multiplier system of the Dedekind $\eta$-function $\eta(\tau):=q^{1/24}\prod_{n\geq 1}(1-q^n)$)
\begin{multline}\label{eqn:Fmodularity}
\mathcal{F}_{h,t,N}\left(\frac{a\tau+b}{c\tau+d}\right) = -e^{-2\pi i  \left(h-\frac{N}{t}\right)^2\frac{a\tau+b}{c\tau+d} }\widetilde{\mu}\left(\frac{ht-N}{2N}\frac{a\left(\frac{8N^2}{t^2}\tau\right)+ \frac{8N^2b}{t^2}}{\frac{ct^2}{8N^2}\left(\frac{8N^2}{t^2}\tau\right)+d},-\frac{1}{2}; \frac{a\left(\frac{8N^2}{t^2}\tau\right)+ \frac{8N^2b}{t^2}}{\frac{ct^2}{8N^2}\left(\frac{8N^2}{t^2}\tau\right)+d}\right)\\
=-e^{-2\pi i  \left(h-\frac{N}{t}\right)^2 \left(\frac{a\tau+b}{c\tau+d}\right) } v(\gamma')^{-3} (c\tau+d)^{\frac{1}{2}}e^{-\pi i \frac{\frac{ct^2}{8N^2}\left(\frac{ht-N}{2N}\left(a\left(\frac{8N^2}{t^2}\tau\right)+\frac{8N^2b}{t^2}\right) +\frac{c\tau+d}{2}\right)^2}{c\tau+d}}
\\
\times \widetilde{\mu}\left(\left(\frac{ht}{2N}-\frac{1}{2}\right)\left(a\left(\frac{8N^2}{t^2}\tau\right)+\frac{8N^2b}{t^2}\right),-\frac{c\tau+d}{2}; \frac{8N^2\tau}{t^2}\right).
\end{multline}
We next use the fact that $a\equiv 1\pmod{4N/t}$ to obtain 
\[
\frac{ht-N}{2N}\left(a\left(\frac{8N^2}{t^2}\tau\right)+\frac{8N^2b}{t^2}\right)\equiv \frac{ht-N}{2N}\frac{8N^2\tau}{t^2} \pmod{\Z \frac{8N^2\tau}{t^2}+\Z},
\]
while $16N^2/t^2\mid c$ and $d\equiv 1\pmod{4N/t}$ imply that 
\[
\frac{c\tau+d}{2}\equiv \frac{1}{2}\pmod{\Z \frac{8N^2\tau}{t^2}+\Z},
\]
Hence by \cite[Theorem 1.11 (1)]{ZwegersThesis}, we have 
\begin{multline}\label{eqn:elliptictransform}
\widetilde{\mu}\left(\frac{ht-N}{2N}\left(a\left(\frac{8N^2}{t^2}\tau\right)+\frac{8N^2b}{t^2}\right),-\frac{c\tau+d}{2}; \frac{8N^2\tau}{t^2}\right)\\
=(-1)^{(a-1)\frac{ht-N}{2N}+\left(h-\frac{N}{t}\right)\frac{4N b}{t}-\frac{c t^2}{16N^2}-\frac{d-1}{2} }e^{\pi i \left((a-1)\frac{ht-N}{2N}+\frac{c t^2}{16N^2}\right)^2 \frac{8N^2\tau}{t^2} + 2\pi i \left((a-1)\frac{ht-N}{2N}+\frac{c t^2}{16N^2}\right)\left(\frac{ht-N}{2N}\frac{8N^2\tau }{t^2}+\frac{1}{2}\right)}\\
\times\widetilde{\mu}\left(\frac{ht-N}{2N}\frac{8N^2\tau }{t^2},-\frac{1}{2}; \frac{8N^2\tau}{t^2}\right).
\end{multline}
The power of $-1$ modifies the multiplier system accordingly.  Plugging back into \eqref{eqn:Fmodularity}, we see that it remains to simplify the exponentials to match the power of $\tau$.

The parameter of the exponential (or rather, the part which involves $\tau$) is $\frac{2\pi i}{c\tau+d}$ times 
\begin{multline}\label{eqn:tosimplify}
-\left(h-\frac{N}{t}\right)^2(a\tau+b)-\frac{ct^2}{16N^2}\left(\frac{ht-N}{2N}\left(a\left(\frac{8N^2}{t^2}\tau\right)+\frac{8N^2b}{t^2}\right) +\frac{c\tau+d}{2}\right)^2\\
+\frac{1}{2}\left((a-1)\frac{ht-N}{2N}+\frac{c t^2}{16N^2}\right)^2 \frac{8N^2\tau}{t^2}(c\tau+d) + \left((a-1)\frac{ht-N}{2N}+\frac{c t^2}{16N^2}\right)\frac{ht-N}{2N}\frac{8N^2\tau }{t^2}(c\tau+d)\\
=-\left(h-\frac{N}{t}\right)^2(a\tau+b)-c(a\tau+b)^2\left(h-\frac{N}{t}\right)^2 -\frac{ct}{4N}\left(h-\frac{N}{t}\right)(a\tau+b)(c\tau+d) - \frac{ct^2}{64N^2} (c\tau+d)^2\\
+(a-1)^2\left(h-\frac{N}{t}\right)^2 \tau(c\tau+d) + \frac{t}{4N} (a-1)\left(h-\frac{N}{t}\right)c\tau(c\tau+d)+ \frac{t^2}{64N^2} c^2\tau (c\tau+d)\\
+2(a-1)\left(h-\frac{N}{t}\right)^2 \tau(c\tau+d) + \frac{t}{4N} \left(h-\frac{N}{t}\right) c\tau(c\tau+d).
\end{multline}
We consider \eqref{eqn:tosimplify} as a polynomial in $h-N/t$ and simplify the coefficients of each power of $h-N/t$.  We first combine and simplify the terms in \eqref{eqn:tosimplify} with $(h-N/t)^2$. Using $ad-bc=1$, these are $(h-N/t)^2$ times
\begin{multline*}
-(a\tau+b) -c(a\tau+b)^2 +(a-1)^2\tau(c\tau+d) +2(a-1)\tau(c\tau+d)\\
 = -a\tau-b -2abc\tau -b^2c +a^2d\tau+d\tau -c\tau^2- 2d\tau= -a\tau-b -2abc\tau -b(ad-1) +a(1+bc)\tau+d\tau -c\tau^2- 2d\tau\\
=-abc\tau-abd +d\tau -c\tau^2- 2d\tau=-(c\tau+d)(\tau+ab).
\end{multline*}
Thus the exponential corresponding to the terms with $(h-N/t)^2$ is 
\[
e^{\frac{2\pi i}{c\tau+d}\left(h-\frac{N}{t}\right)^2 (c\tau+d)(-\tau-ab)} = e^{-2\pi i \left(h-\frac{N}{t}\right)^2 \tau} e^{-2\pi i  \left(h-\frac{N}{t}\right)^2 ab}.
\]
The first factor is precisely the factor in front of $\mathcal{F}_{h,t,N}$ and the second contributes to the multiplier system.  

We next simplify the terms in \eqref{eqn:tosimplify} with $h-N/t$. These give
\[
\left(h-\frac{N}{t}\right)(c\tau+d)\frac{ct}{4N}\left(-(a\tau+b) +(a-1)\tau +\tau\right) = -b\left(h-\frac{N}{t}\right)(c\tau+d)\frac{ct}{4N}.
\]
The resulting exponential contributes to the multiplier system since the factor $c\tau+d$ cancels. 

Finally, we see directly that the terms in \eqref{eqn:tosimplify} which are constant when considered as a polynomial in $h-N/t$ cancel. Therefore, the simplification of \eqref{eqn:tosimplify} yields that the exponential is
\begin{equation}\label{eqn:expsimplified}
e^{-2\pi i \left(h-\frac{N}{t}\right)^2 \tau} e^{-2\pi i  \left(h-\frac{N}{t}\right)^2 ab}e^{-2\pi i b\left(h-\frac{N}{t}\right)\frac{ct}{4N}}.
\end{equation}
Altogether, plugging \eqref{eqn:elliptictransform} and \eqref{eqn:expsimplified} into \eqref{eqn:Fmodularity} (note that in the simplification we left out one exponential term in \eqref{eqn:elliptictransform} because it was independent of $\tau$) yields 
\begin{multline}\label{eqn:Fmodularityfull}
\mathcal{F}_{h,t,N}\left(\frac{a\tau+b}{c\tau+d}\right) =  v(\gamma')^{-3} (c\tau+d)^{\frac{1}{2}}(-1)^{(a-1)\frac{ht-N}{2N}+\left(h-\frac{N}{t}\right)\frac{4N b}{t}-\frac{c t^2}{16N^2}-\frac{d-1}{2} }\\
\times e^{-2\pi i  \left(h-\frac{N}{t}\right)^2 ab}e^{-2\pi i b\left(h-\frac{N}{t}\right)\frac{ct}{4N}}e^{\pi i \left((a-1)\frac{ht-N}{2N}+\frac{c t^2}{16N^2}\right)}\mathcal{F}_{h,t,N}(\tau).
\end{multline}
We see from \eqref{eqn:Fmodularityfull} that $\mathcal{F}_{h,t,N}$ has the desired modularity properties.

We next compute the image under $\xi_{1/2}$.  Since the $\mu$-function is holomorphic on the upper half-plane, it is annihilated by $\xi_{1/2}$.  Therefore, plugging in the definition \eqref{eqn:tildemudef} of $\widetilde{\mu}$, we have 
\[
\xi_{\frac{1}{2}}\left(\mathcal{F}_{h,t,N}(\tau)\right) =-\frac{1}{2i}\xi_{\frac{1}{2}}\left(e^{-2\pi i \left(h-\frac{N}{t}\right)^2\tau}R\left(\frac{8N^2}{t^2} \left(\frac{ht}{2N}-\frac{1}{2}\right)\tau+\frac{1}{2};\frac{8N^2\tau}{t^2}\right)\right).
\]
Noting that we have 
\[
\frac{\im\left(\frac{8N^2}{t^2} \left(\frac{ht}{2N}-\frac{1}{2}\right)\tau+\frac{1}{2}\right)}{\im\left(\frac{8N^2\tau}{t^2}\right)} = \frac{ht}{2N}-\frac{1}{2},
\]
we then employ \cite[Theorem 1.16]{ZwegersThesis} to rewrite this as 
\begin{equation}\label{eqn:xiEichler}
\xi_{\frac{1}{2}}\left(\mathcal{F}_{h,t,N}(\tau)\right) =-\frac{1}{2i}\xi_{\frac{1}{2}}\left(\int_{-\frac{8N^2\overline{\tau}}{t^2}}^{i\infty} \frac{g_{\frac{ht}{2N},0}(z)}{\sqrt{-i\left(z+\frac{8N^2 \tau}{t^2}\right)}} dz\right),
\end{equation}
where 
\[
g_{a,b}(\tau):=\sum_{\nu\in a+\Z} \nu e^{\pi i \nu^2 \tau + 2\pi i b\nu}.
\]
The remaining integral is what is known as a non-holomorphic Eichler integral, and is easily evaluated by the Fundamental Theorem of Calculus as 
\[
\xi_{\frac{1}{2}}\left(\int_{-\frac{8N^2\overline{\tau}}{t^2}}^{i\infty} \frac{g_{\frac{ht}{2N},0}(z)}{\sqrt{-i\left(z+\frac{8N^2 \tau}{t^2}\right)}} dz\right)= -2iv^{\frac{1}{2}}\frac{8N^2}{t^2} \frac{g_{\frac{ht}{2N},0}\left(\frac{8N^2\tau }{t^2}\right)}{\frac{2N}{t}  \sqrt{-2i\!\left(\tau- \overline{\tau}\right)}}= -2i \frac{2N}{t} g_{\frac{ht}{2N},0}\left(\frac{2N}{t}\tau\right).
\]
Therefore \eqref{eqn:xiEichler} becomes
\[
\xi_{\frac{1}{2}}\left(\mathcal{F}_{h,t,N}(\tau)\right) =\frac{2N}{t}g_{\frac{ht}{2N},0}\left(\frac{8N^2}{t^2}\tau\right).
\]
We finally rewrite 
\[
\frac{2N}{t}g_{\frac{ht}{2N},0}\left(\frac{8N^2}{t^2}\tau\right)= \frac{2N}{t}\sum_{\nu\in \frac{ht}{2N}+\Z} \nu e^{\frac{8\pi i N^2 \nu^2 \tau}{t^2}} =   \sum_{\nu\in h+\frac{2N}{t}\Z} \nu e^{2\pi i \nu^2\tau} = \vartheta_{h,t,N}(\tau). 
\]
\end{proof}

In order to prove Proposition \ref{prop:octagonal}, we are particularly interested in the case of $N=3$ and $h=2$.  It turns out that congruence conditions immediately rule out all of the possible unary theta functions except for the form 
\begin{equation}\label{eqn:thetaoctagonal}
\vartheta_{\chi_{-3}}(\tau):=\sum_{n\in\Z} \chi_{-3}(n) n e^{2\pi i n^2\tau},
\end{equation}
where $\chi_d(n):=\left(\frac{d}{n}\right)$ is the usual Kronecker--Jacobi character (also known as the extended Legendre symbol).  We rewrite this form in the notation from this paper as follows.
\begin{lemma} \label{lem:thetarewrite}
We have 
\[
\vartheta_{\chi_{-3}}(\tau) = \vartheta_{2,1,3}\left(\frac{\tau}{4}\right).
\]
\end{lemma}
\begin{remark}
By Theorem \ref{thm:Zwegersfull}, Lemma \ref{lem:thetarewrite} together with the chain rule implies that
\[
\xi_{\frac{1}{2}}\left(\mathcal{F}_{2,1,3}\left(\frac{\tau}{4}\right)\right)=\frac{1}{4}\vartheta_{\chi_{-3}}(\tau).
\]
\end{remark}
\begin{proof}
We compute 
\begin{multline*}
\vartheta_{\chi_{-3}}(\tau) = \sum_{n\in\Z} (3n+1) n e^{2\pi i (3n+1)^2\tau}-\sum_{n\in\Z} (3n-1) n e^{2\pi i (3n-1)^2\tau}\\
= \sum_{n\in\Z} (3n+1) n e^{2\pi i (3n+1)^2\tau}-\sum_{n\in\Z} (-3n-1) n e^{2\pi i (-3n-1)^2\tau}\\
=2\sum_{n\equiv 1\pmod{3}} n e^{2\pi i n^2\tau} = \sum_{n\equiv 2\pmod{6}} n e^{2\pi i n^2\frac{\tau}{4}} =\vartheta_{2,1,3}\left(\frac{\tau}{4}\right).
\end{multline*}
\end{proof}

\section{An application to lattice theory}\label{sec:applications}
\subsection{An application}
To motivate this study, we first prove Theorem \ref{thm:shiftedlattice}.
\begin{proof}[Proof of Theorem \ref{thm:shiftedlattice}]
We decompose $\Theta_{L+\nu}$ as an Eisenstein series $E$, a unary theta function, and a cusp form $g$ which is orthogonal to unary theta functions.  Since the unary theta function is trivial by assumption, we have 
\[
\Theta_{L+\nu} = E+g.
\]
We then compare the coefficients of $E+g$.  Since every element of $M\Z+r$ is primitively represented locally, the local densities increase as a function of $n$.  The product of the local densities were shown in \cite{vanderBlij} (and independently in \cite{Shimura2004}) to be the Fourier coefficients of $E$, paralleling the famous Siegel--Weil formula.  Since $\nu\in\Q L$, there exists $R\in\N$ for which $R\nu\in L$.  Note further that (denoting the localization at the prime $p$ by $L_p:=L\otimes \Q_p$) for each prime $p\nmid R$, we have $\nu\in L_p$ (because $R$ is invertible in $\Q_p$)  Therefore
\[
L_p+\nu = L_p.
\]
In other words, the local density at $p$ for $L+\nu$ and for $L$ agree.  Denoting the local densities for $L+\nu$ by $\beta_p$ and the local densities of $L$ by $\alpha_p$, we have 
\[
\prod_{p} \beta_p = \frac{\prod_{p\mid R} \beta_p}{\prod_{p\mid R} \alpha_p}\prod_{p} \alpha_p.
\]
The product $\prod_{p}\alpha_p$ is known to be a (Hurwitz) class number for an imaginary quadratic field (see \cite[Theorem 86]{Jones}) and these are known to grow faster than $n^{\frac{1}{2}-\varepsilon}$ by Siegel's \cite{Siegel} famous (ineffective) lower bound for the class numbers.  On the other hand, Duke \cite{Duke} has shown that the coefficients of $g$ grow slower than $n^{3/7+\varepsilon}$.  Therefore, the coefficients of $E$ are the main asymptotic term and they are positive.  For $n$ sufficiently large the coefficient must be positive, yielding the claim. 
\end{proof}
It is worth noting that the Fourier coefficients of the unary theta function grow at the same rate as the coefficients of the Eisenstein series.  In other words, when the unary theta function is not trivial, it is often the case that the set investigated in Theorem \ref{thm:shiftedlattice} is actually infinite.  One such example is worked out in \cite[Theorem 1.5]{HaenschKane} with an applications to sums of polygonal numbers, and a proposed algebraic explanation for this behavior involving the spinor genus of $L+\nu$ is given in \cite[Conjecture 1.3]{HaenschKane}. 

\subsection{An individual case}
In individual cases, one may combine Theorem \ref{thm:shiftedlattice} with Theorem \ref{thm:Zwegersfull} to show that certain quadratic polynomials are almost universal. We demonstrate one such example in Proposition \ref{prop:octagonal}. 
\begin{proof}[Proof of Proposition \ref{prop:octagonal}]
Let $L+\nu$ be the corresponding shifted lattice. By Theorem \ref{thm:shiftedlattice}, it suffices to show that the inner product of $\Theta_{L+\nu}$ against all theta functions in the same space is trivial. For the diagonal lattice corresponding to the quadratic form $Q(x,y,z)=x^2+3y^2+3z^2$, an inclusion-exclusion argument implies that (recalling that $\Theta(\tau)=\sum_{n\in\Z} q^{n^2}$)
\[
\Theta_{L+\nu}(\tau) = \left(\Theta(\tau)-\Theta(9\tau)\right)\left(\Theta(3\tau)-\Theta(27\tau)\right)\left(\Theta(3\tau)-\Theta(27\tau)\right),
\]
from which one sees that $\Theta_{L+\nu}$ is actually a weight $3/2$ modular form on $\Gamma_0(108)$. Specifically, in Shimura's notation, we have 
\[
\Theta_{L+\nu}(\tau) = \theta\left(6\tau;\left(\begin{matrix} 3\\ 9\\ 9\end{matrix}\right), \left(\begin{matrix} 3&0 &0\\0&9&0\\ 0&0&9\end{matrix}\right), 9,1\right), 
\]
where (for $A$ a symmetric $n\times n$ matrix, $h\in\Z^n$ satisfying $Ah\in N\Z^n$, and $P$ a spherical function)
\begin{equation}\label{eqn:Shimuradef}
\theta(\tau;h,A,N,P):=\sum_{\substack{x\in\Z^n\\ x\equiv h\pmod{N}}} P(x) e^{\frac{2\pi i \tau }{2N^2} {^t x} Ax}.
\end{equation}
Here $^tx$ denotes the transpose of $x$. We write $h_L$ and $A_L$ for the corresponding vector and lattice in our case and omit $P=1$ in the notation in the following.

A straightforward check of congruence conditions implies that the only relevant theta function is $\vartheta_{\chi_{-3}}$ defined in \eqref{eqn:thetaoctagonal}. By Lemma \ref{lem:thetarewrite}, Theorem \ref{thm:Zwegersfull}, and Theorem \ref{thm:BF}, it suffices to show that, for $\Gamma=\Gamma_0(108)$,
\begin{equation}\label{eqn:sumSgamma}
\sum_{\varrho\in \mathcal{S}_{\Gamma}}\sum_{n\geq 0}c_{\Theta_{L+\nu},\varrho}(n) c_{\mathcal{F}_{2,1,3}(\tau/4),\varrho}(-n)=0,
\end{equation}
where we abuse notation to write $c_{\mathcal{F}_{2,1,3}(\tau/4),\varrho}(-n)$ as the $(-n)$th coefficient of $\mathcal{F}_{2,1,3}(\tau/4)$.  In order to compute the expansions at other cusps, we apply $S$ and $T$ repeatedly (using Lemma \ref{lem:matrixrewrite}) and note that \cite[Theorem 1.11 (2)]{ZwegersThesis} yields the fully modularity properties of $\mathcal{F}_{2,1,3}(\tau/4)$ as a vector-valued modular form, while $\Theta_{L+\nu}$ behaves as a vector-valued modular form on the full modular group by \cite[(2.4) and (2.5)]{Shimura}. Specifically, we have (for arbitrary $h$ satisfying $A_L h\in 9\Z^3$)
\begin{align*}
\theta\!\left(-\frac{1}{z};h,A_L,9\right) = \sum_{\substack{k\pmod{9}\\ A_Lk\equiv 0\pmod{9}}} e^{\frac{2\pi i}{27}\left( k_1h_1 + 3 k_2h_2+3k_3h_3\right)}\theta\!\left(z; k,A_L,9\right),\\
\theta(z+2;;h,A_L,9) = e^{\frac{2\pi i }{27}\left(h_1^2+3h_2^2+3h_3^2\right)} \theta\!\left(z;h,A_L,9\right).
\end{align*}
Note that the restriction $A_Lh\equiv 0\pmod{9}$ is equivalent to $3\mid h_1$, so the exponential in the first identity may be simplified as 
\[
e^{\frac{2\pi i}{9}\left( \frac{k_1h_1}{3} + k_2h_2+k_3h_3\right)}
\]
and the exponential in the second identity may be simplified as 
\[
e^{\frac{2\pi i }{9}\left(\frac{h_1^2}{3}+h_2^2+h_3^2\right)}.
\]
Since the only terms contributing to the sum in \eqref{eqn:sumSgamma} are the principal parts (the terms where the power of $q$ is negative) of the expansions around each cusp of $\mathcal{F}_{2,1,3}$, we only need to compute a few Fourier coefficients for each of the components of the vector-valued modular forms corresponding to $\Theta_{L+\nu}$ and $\mathcal{F}_{2,1,3}$. A computer check then verifies \eqref{eqn:sumSgamma}, yielding the claim in the proposition.
\end{proof}

\end{document}